\documentclass[a4paper]{amsart}

\usepackage{amsthm,amsfonts,amsmath, amssymb}
\usepackage{color}
\def\r{\mathbb{R}}

\usepackage{mathpazo}

\usepackage{geometry}                
\usepackage{anysize}
\marginsize{3cm}{3cm}{3cm}{3cm}

\renewcommand{\baselinestretch}{1.25}
\usepackage[sort&compress,square,comma,numbers]{natbib}

\usepackage{graphicx,float,enumerate}
\usepackage[ruled,boxed,commentsnumbered,norelsize]{algorithm2e}
\usepackage{verbatim}
\usepackage{url}
\usepackage{epstopdf}
\usepackage[notref,notcite,final]{showkeys}

\usepackage[capitalise]{cleveref}

\newtheorem{theorem}{Theorem}
\newtheorem{corollary}[theorem]{Corollary}
\newtheorem{lemma}[theorem]{Lemma}

\newtheorem{rmk}[theorem]{Remark}

\crefname{conjecture}{Conjecture}{Conjectures}

\crefname{question}{Question}{Questions}

\theoremstyle{definition}

\theoremstyle{remark}

\usepackage{mathtools}

\DeclareMathOperator{\ver}{vert}

\newcommand{\h}{\frac{1}{2}}
\begin{document}

\date{\today}
\title{Minimum number of edges of polytopes with $2d+2$  vertices}
\author{Guillermo Pineda-Villavicencio}
\author{Julien Ugon}
\author{David Yost}
\address{School of Information Technology, Deakin University, Geelong, Victoria 3220, Australia}
\email{\texttt{guillermo.pineda@deakin.edu.au}}
\address{School of Information Technology, Deakin University, Geelong, Victoria 3220, Australia}
\email{\texttt{julien.ugon@deakin.edu.au}}
\address{Centre for Informatics and Applied Optimisation, Federation University, Mt. Helen, Victoria 3350, Australia}
\email{\texttt{d.yost@federation.edu.au}}


%

\begin{abstract}
We define an analogue of the cube and an analogue of the 5-wedge in higher dimensions, each with $2d+2$ vertices and $d^2+2d-3$ edges. We show that these two are the only minimisers of the number of edges, amongst $d$-polytopes with $2d+2$ vertices, when $d=6$ or $d\ge8$. For $d=4, 5$ or 7, we also characterise the minimising polytopes; there are four sporadic examples in these dimensions. We announce a partial solution to the corresponding problem for polytopes with $2d+3$  vertices.
\end{abstract}

\maketitle

\section{Background: excess, taxonomy and decomposability}

This paper is concerned with  graphs of polytopes with not too many vertices. Throughout, we will denote the number of vertices and edges of a polytope $P$ by $v(P)$ and $e(P)$ respectively, or simply by $v$ and $e$ if $P$ is clear from the context. The set of vertices is denoted as usual by ${\rm Vert}(P)$. Different letters will be used for the names of individual vertices. The dimension of the ambient space is denoted by $d$.

In 1967, Gr\"unbaum \cite[Sec. 10.2]{Gru03} made a conjecture  concerning the minimum number of edges of $d$-polytopes with $v\le2d$ vertices, and confirmed it for $v\le d+4$. In \cite{PinUgoYos19}, we confirmed it for $v\leq2d$, and also  characterised the minimising polytope, which is unique for each $v$ (up to combinatorial equivalence). We also found the corresponding results for polytopes with $2d+1$ vertices. We extend this program here by calculating the minimum number of edges of polytopes with $2d+2$ vertices, also characterising the minimising polytopes. In the final section, we consider the case of $2d+3$ or more vertices.

An important concept in resolving Gr\"unbaum's conjecture was the {\it excess degree}. Recall that the {\it degree} of any vertex is the number of edges incident to it;  this cannot be less than the dimension of the ambient polytope. We defined the {\it excess degree} of a vertex $u$ as $\deg u-d$; thus a vertex is {\it simple} if its excess degree is zero. We then define the {\it excess} of a $d$-polytope $P$, denoted $\xi(P)$, as $\sum_{u\in \ver P}(\deg u-d)$, i.e. the sum of the excess degrees of its vertices. Thus a polytope is \textit{simple}, i.e. every vertex is simple, if $\xi(P)=0$. A vertex is {\it non-simple} in a $d$-polytope $P$ if its degree in $P$ is at least $d+1$. A polytope with at least one non-simple vertex is called {\it non-simple}. It is easy to see that
$$\xi(P)=2e(P)-dv(P).\eqno{(1)}$$

A fundamental result about the excess degree is that it cannot take arbitrary values \cite[Theorem 3.3]{PinUgoYos18}.

\begin{theorem}\label{thm:excess} Let $P$ be a non-simple $d$-polytope. Then $\xi(P)\ge d-2$. \end{theorem}

Recall that the Minkowski sum of two polytopes $Q,R$ is simply $Q+R=\{q+r:q\in Q, r\in R\}$. A {\it prism} based on a facet $F$ is the Minkowski sum of $F$ and a line segment, or any polytope combinatorially equivalent to it. The {\it simplicial $d$-prism} is any prism whose base is a $(d-1)$-simplex; we will often refer to these simply as prisms. Any $d$-dimensional simplicial prism has $2d$ vertices, $d^2$ edges, and $d+2$ facets.
For $m,n>0$, the polytope $\Delta_{m,n}$   is be defined as the Minkowski sum of an $m$-dimensional simplex and an $n$-dimensional simplex, lying in complementary subspaces. It is easy to see that it has dimension $m+n$, $(m+1)(n+1)$ vertices, $m+n+2$ facets, $\h(m+n)(m+1)(n+1)$ edges, and is simple. For $n=1$,  $\Delta_{d-1,1}$ is simply a prism. Being simple, all the polytopes just described have excess degree 0.

\begin{rmk}[Facets of $\Delta_{m,n}$]\label{rmk:Delta-Facets} The facets of the $(m+n)$-polytope $\Delta_{m,n}$, $m+n+2$ in total, are:
\begin{itemize}
\item $m+1$ copies of $\Delta_{m-1,n}$,
\item $n+1$ copies of $\Delta_{m,n-1}$.
\end{itemize}
In particular, $\Delta_{m,n}$ contains no simplex facets at all if $m\ge2$ or $n\ge2$, and no two facets of a simplex are disjoint.
\end{rmk}

A {\it triplex} is defined as a multifold pyramid over a simplicial prism. More precisely a $(k,d-k)$-triplex, denoted $M_{k,d-k}$ is a $(d-k)$-fold pyramid over the simplicial $k$-prism. We recall the quadratic polynomial defined by Gr\"unbaum \cite[p 184]{Gru03},
$$\phi(v,d)={d+1\choose2}+{d\choose2}-{2d+1-v\choose2}={v\choose2}-2{v-d\choose2}.$$

Note also the equivalent expression for $\phi$,
$$\phi(d+k,d)=\h d(d+k)+\h(k-1)(d-k).$$

The following result verifies Gr\"unbaum's conjecture about the minimum number of edges of $d$-polytopes with up to $2d$ vertices.

\begin{theorem}\label{thm:triplexes}\cite[Theorem 7]{PinUgoYos19} Let $P$ be a $d$-polytope with $d+k$ vertices, where $1\le k\le d$. Then $P$ has at least $\phi(d+k,d)={d\choose 2}-{k\choose2}+kd$ edges, equivalently $P$ has excess degree at least $(k-1)(d-k)$. Furthermore, equality is obtained only if $P$ is a $(k,d-k)$-triplex, i.e~ a $(d-k)$-fold pyramid over the simplicial $k$-prism.
\end{theorem}

A {\it missing edge} in a polytope is a pair of distinct vertices with no edge between them. \cref{thm:triplexes} then says that a $d$-polytope with $v\le2d$ vertices has at most $2{v-d\choose2}$ missing edges, and that this maximum is attained by the appropriate triplex.

For simplicial polytopes, the well known Lower Bound Theorem gives a stronger conclusion, without a restriction on the number of vertices.

\begin{theorem}\label{thm:LBT}\cite{Bar73} Let $P$ be a $d$-polytope with $v=d+k$ vertices, and suppose that every facet of  $P$ is a simplex. Then $P$ has  $dv-{d+1\choose 2}$ edges, equivalently $P$ has excess degree at least $(k-1)d$.
\end{theorem}

In describing a polytope, it is enough to know all the vertex-facet incidences; this determines the entire face lattice. We need to be familiar with some important examples. Truncating a simple vertex of any polytope clearly yields a new polytope with $d-1$ more vertices than the original, but the same excess degree; thus the number of edges will increase by $d\choose2$.	
A pentasm (needed in case 3 of our main theorem) can be defined \cite[p. 2015]{PinUgoYos18} as a $(2,d-2)$-triplex with a simple vertex truncated. It has $2d+1$ vertices, $d^2+d-1$ edges, and hence excess degree $d-2$. Its facets are $d+3$ in number: $d-2$ pentasms of  dimension $d-1$; two prisms; and three simplices.
Another way to view the pentasm is as the convex hull of two disjoint faces: a pentagon  and a $(d-2)$-dimensional prism.

\begin{theorem}\label{thm:2d+1}\cite[Thm. 13(iii)]{PinUgoYos19} Let $P$ be a $d$-polytope with $2d+1$ vertices, where $d\ge5$. Then $P$ has at least $d^2+d-1$ edges, with equality only if $P$ is a pentasm.
\end{theorem}

Additional minimisers of the number of edges (of $d$-polytopes with $2d+1$ vertices) appear when $d=3$ or 4; these will be discussed shortly.

We now recall from \cite[\S2.2]{PinUgoYos19} some examples of polytopes with few vertices and edges, which occur repeatedly in our work.

\begin{figure}[h]
\includegraphics[scale=1]{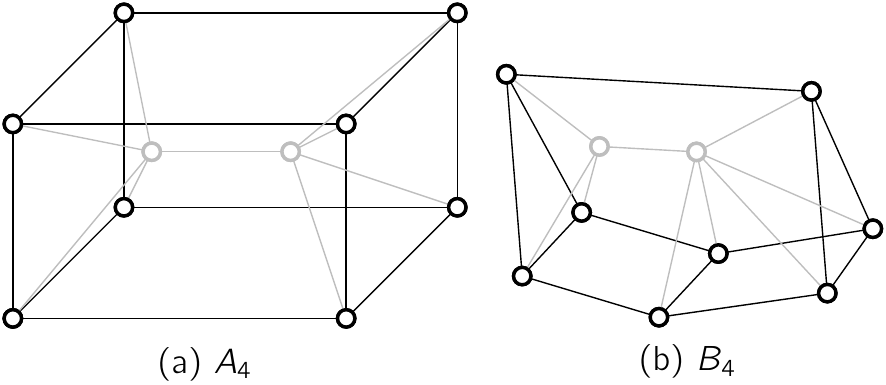}
\caption{Polytopes $A_4$ and $B_4$.}
\label{a4b4}
\end{figure}

Denote by $A_d$ a polytope obtained by truncating a nonsimple vertex of  a $(2,d-2)$-triplex. This polytope can be also realised as a prism over a copy of $M_{2,d-3}$. It has $2d+2$ vertices and excess degree $2d-6$ (\cref{a4b4}).

\begin{rmk}[Facets of $A_d$] \label{rmk:Ad-Facets}
The facets of the $d$-polytope $A_d$, $d+3$ in total, are as follows.
\begin{itemize}
\item $d-3$ copies of $A_{d-1}$,
\item 4 simplicial prisms, and
 \item 2 copies of $M_{2,d-3}$.
\end{itemize}
\end{rmk}

Denote by $B_d$  a polytope obtained by truncating a simple vertex of a $(3,d-3)$-triplex. The polytope $B_3$ is the well known {\it 5-wedge}. The polytope $B_d$ can  be also realised as the convex hull of $B_3$ and a simplicial $(d-3)$-prism $R$ where each vertex of one copy of the $(d-4)$-simplex in $R$ is adjacent to each of the three vertices in a triangle of $B_3$ and  each vertex of the other copy of the $(d-4)$-simplex in $R$ is adjacent to each of the remaining five vertices of $B_3$.  Besides, it has $2d+2$ vertices and excess degree $2d-6$ (\cref{a4b4}).

\begin{rmk}[Facets of $B_d$]\label{rmk:Bd-Facets}
The facets of the $d$-polytope $B_d$, $d+3$ in total, are as follows.
\begin{itemize}
\item $d-3$ copies of $B_{d-1}$,
\item 2 simplices,
\item 1 simplicial prism,
\item 1 copy of $M_{2,d-3}$, and
\item 2 pentasms.
\end{itemize}
\end{rmk}

\begin{rmk}[Similarity of $A_d$, $B_d$]\label{rmk:AdBd}
There is a certain commonality in the structure of $A_{d}$ and $B_{d}$. In both cases, the polytope  can be described as the convex hull of three disjoint faces, two $(d-5)$-dimensional simplices $S_1$ and $S_2$ (whose convex hull constitutes a prism), and a simple 3-face (either a cube or a 5-wedge). The vertices of the 3-face can be partitioned into two subsets  $Q_1$ and $Q_2$, in such a way that a vertex in $S_i$ is adjacent to a vertex in $Q_j$ if and only if $i=j$. In the case of the cube,  $Q_1$ and $Q_2$ correspond to two opposite faces. In the case of the 5-wedge,  $Q_1$ corresponds to a triangular face and $Q_2$ corresponds to the other triangular face, together with the quadrilateral with which it shares an edge.
\end{rmk}

We have presented the structure of $A_d$ and $B_d$ in some detail because, in most dimensions, these two examples are the minimisers of the number of edges, amongst all $d$-polytopes with $2d+2$ vertices. As detailed in \cref{thm:2d+2} below, there are some exceptions in low dimensions, which we now describe.

If two simple vertices of a $d$ polytope are adjacent, truncating that edge will produce a polytope with $2d-4$ more vertices and the same excess degree.
Denote by $C_d$ a polytope obtained by truncating such a {\it simple edge}, of a  $(2,d-2)$-triplex. It has $3d-2$ vertices and excess degree $d-2$. Obviously $C_2$ is just another quadrilateral.

Denote by $\Sigma_d$ a certain polytope which is combinatorially equivalent to the convex hull of
 \[\{0,e_1,e_1+e_k,e_2,e_2+e_k,e_1+e_2,e_1+e_2+2e_k: 3\le k\le d\},\]
where $\{e_i\}$ is the standard basis of $\mathbb{R}^d$. It is easily shown to have $3d-2$ vertices; of these, one has excess degree $d-2$, and the rest are simple. It can be expressed as the Minkowski sum of two
simplices. For consistency, we can also define $\Sigma_2$ as a quadrilateral.

Recall from \cite[p. 2017]{PinUgoYos18} that $J_d$ is the simple polytope obtained by slicing one
vertex of a simplicial $d$-prism; it clearly has $3d-1$ vertices. Of course $J_2$ is just a
pentagon and $B_3$ coincides with $J_3$. The facet sof $J_d$ are $d-1$ copies of $J_{d-1}$,
two  prisms, and two simplices.

Let us also introduce here the polytope $N_d$ obtained by truncating a simple vertex of a triplex $M_{d-1,1}$. It also has $3d-2$ vertices:  one with excess degree $d-2$, with the rest being simple. Its facets are one prism, one copy of $J_d$, one copy of $M_{d-2,1}$, two simplices and $d-2$ copies of $N_{d-1}$. Note that $N_3$ is a pentasm, and $N_4=B_4$. For $d\ge5$, $N_d$ is distinct from all the examples just described. This can also be considered as a basic example, but we will not see it again until the end of the paper.

Now we present some technical results which will be necessary later.

It is well known that any simple $d$-polytope, other than a simplex or prism, has at least $3d-3$ vertices. This follows easily from the $g$-theorem \cite[\S8.6]{Zie95}, but elementary arguments are also available. More precisely, we have the following classification, which is a rewording of \cite[Lemma 10(ii)]{PrzYos16}.

\begin{lemma}\label{lem:simple} Any simple $d$-polytope with strictly less than $3d-1$ vertices is either a simplex, a prism, $\Delta_{2,d-2}$ or $\Delta_{3,3}$. In particular, for every $d\ne 6$, the smallest vertex counts of simple $d$-polytopes are $d+1,2d,3d-3$ and $3d-1$. In dimension 6 only, there is also a simple polytope with $3d-2$ vertices. The only one of these which contains two disjoint simplex facets is the prism.
\end{lemma}

The next result \cite[Lemma 2.5]{PinUgoYos18} is surprisingly useful.

\begin{lemma}\label{lem:simpleVertexOutside}
Let $P$ be a polytope, $F$ a facet of $P$ and $u$ a nonsimple vertex of $P$ which is contained in $F$. If $u$ is adjacent to a simple vertex $x$ of $P$ in $P\setminus F$, then $u$ must be adjacent to another vertex in $P\setminus F$, other than $x$.
\end{lemma}

A very useful tool for us is the following concept: a polytope $P$ is called {\it decomposable} if it can be expressed as the (Minkowski) sum of two polytopes which are not similar to it, i.e. not obtainable from $P$ just by translation and scaling. Inevitably, all other polytopes are described as  {\it indecomposable.} We refer to \cite{PrzYos08} and the references therein for a more detailed discussion of this topic. Kallay \cite{Kal82} showed that decomposability of a polytope can often be decided from properties of its graph. He introduced the concept of a {\it geometric graph}, as any graph $G$ whose vertex set $V$ is a subset of $\r^d$ and in which every edge is a line segment
joining members of $V$; we find it convenient to add the restriction that no three vertices are collinear. Such a graph  need not be the edge graph of any polytope. He then extended the notion of decomposability to geometric
graphs in a consistent manner. We omit his definition; the important point
\cite[Theorem 1]{Kal82} is that a polytope is indecomposable as just defined if and only if its edge graph is indecomposable in his sense.

A strategy for proving indecomposability of a polytope is to prove that certain basic geometric graphs are indecomposable, and then to build up from them to deduce that the entire skeleton of our polytope is indecomposable. As in \cite[p. 171]{PrzYos16}, let us say that a geometric graph $G=(V,E)$ is a simple extension of a  geometric graph  $G_0=(V_0,E_0)$ if $G_0$ is a subgraph of $G$, $V\setminus V_0$ is just one vertex, and  $E\setminus E_0$ comprises two (or more) edges containing that vertex. The following result summarizes everything we need in the sequel; we have not included stronger known statements about decomposability.

\begin{theorem}\label{decomp}

\begin{enumerate}[(i)]
\item If $G$ is a simple extension of $G_0$, and $G_0$ is an indecomposable geometric graph, then $G$ is also indecomposable.
\item A single edge is indecomposable; any triangle is indecomposable.
\item A geometric graph isomorphic to the complete bipartite graph $K_{2,3}$  is decomposable if and only if it lies in a plane.
\item A polytope $P$ is indecomposable, if (and only if) its graph contains an indecomposable subgraph $G$  whose vertex set contains at least one vertex from every facet of $P$.
\item If $P$ is a pyramid, then it is indecomposable.
\item Any $d$-polytope with $2d$ or fewer vertices, other than the prism, is indecomposable.
\end{enumerate}
\end{theorem}

\begin{proof} (indication)
(i) See \cite[Prop. 1]{PrzYos16}

(ii) This is easy for an edge; the case of a triangle then follows from (i).

(iii) This is proved without statement in \cite[Example 12]{PrzYos08}.

(iv) Proofs may be found in \cite{PrzYos08,PrzYos16} and elsewhere.

(v) Consider any edge containing the apex of the pyramid. This is an indecomposable subgraph which touches every facet, and the conclusion follows from (iv).

(vi) See \cite[Theorem 9]{PrzYos16}.
\end{proof}

In the other direction, the following  sufficient condition  will be useful to us several times. It is due to Shephard; for another proof, see \cite[Prop. 5]{PrzYos16}.  We  say \cite[p. 30]{PinUgoYos19} that a facet $F$ of a polytope $P$ has {\it Shephard's property} if for every vertex $u \in F$, there exists exactly one edge in $P$ that is incident to $u$ and does not lie in $F$. We also say that a
polytope is a {\it Shephard polytope} if it has at least one facet with Shephard's property.

\begin{theorem}
[{\cite[Result (15)]{She63}}] \label{shp} If a polytope $P$ has a facet $F$ with   Shephard's property, and there are at least two vertices outside $F$, then $P$ is decomposable. In particular, any simple polytope other than a simplex is decomposable.
\end{theorem}

\section{Polytopes with $2d+2$ vertices}

As in \cite{PinUgoYos19}, we define  the set $E(v,d)=\{e:$ there is a $d$-polytope with $v$ vertices and $e$ edges$\}$. \cref{thm:triplexes} asserts, for each fixed $k\le d$, that $\min E(d+k,d)=\h d(d+k)+\h(k-1)(d-k)$, and that the triplex $M_{k,d-k}$ is the unique minimiser. So $\min E(v,d)$ is known, whenever $v\le2d$.

\cref{thm:2d+1} asserts that $\min E(2d+1,d)=d^2+d-1$ for $d\ge5$, and that the pentasm is the unique minimiser. For low dimensions, some sporadic examples occur. For $d=3$, it is easy to check that there is a second minimiser, namely $\Sigma_3$. For $d=4$, the pentasm is the only polytope with nine vertices and 19 edges, but $\Delta_{2,2}$ has nine vertices and 18 edges.

Similarly, we will show here that $\min E(2d+2,d)=d^2+2d-3$ for all $d\ge3$ except $d=5$. (It is well known that $\min E(12,5)=30<32=5^2+2\times5-3$, and that $\Delta_{2,3}$ is the only 5-polytope with 12 vertices and 30 edges. It also follows from the Excess Theorem (\cref{thm:excess}) that no 5-polytope has 12 vertices and 31 edges.) Furthermore, we show that, for all $d\ge3$ except $d=4$ and $d=7$, the only polytopes with $2d+2$ vertices and $d^2+2d-3$ edges are the polytopes $A_d$  and $B_d$  defined above.

\begin{theorem}\label{thm:2d+2} For $d\ge3$, the only $d$-polytopes with $2d+2$ vertices and precisely $d^2+2d-3$  edges,  equivalently with excess degree $2d-6$, are as follows.
\begin{enumerate}[(i)]
\item For $d=3$, $d=5$, $d=6$ and all $d\ge8$, only the two polytopes $A_d$ and $B_d$.
\item For $d=4$, the four polytopes $A_4$, $B_4$, $C_4$ and $\Sigma_4$.
\item For $d=7$, the three polytopes $A_7$, $B_7$ and the pyramid over $\Delta_{2,4}$.
\end{enumerate}
Moreover, the 5-polytope $\Delta_{2,3}$ is the only polytope of any dimension with  $2d+2$ vertices and strictly fewer than $d^2+2d-3$  edges.
 \end{theorem}

The case $d=3$ of \cref{thm:2d+2} is easy to check; one may also consult catalogues \cite{BriDun73,Fed74}. The case $d=4$ was established in \cite[Theorem 6.1]{PinUgoYos18}. Some arguments in the sequel are simplified by considering only the case $d\ge5$.
We establish several special cases first in order to streamline the proof. Some of them are of independent interest.

\begin{lemma}\label{lem:pyramid} Let $P$ be a $d$-polytope with $2d+2$ vertices and no more than $d^2+2d-3$  edges. If $P$ is a pyramid, then $d=7$ and the base of $P$ is  $\Delta_{2,4}$.
\end{lemma}

\begin{proof}
Our hypothesis amounts to saying that $P$ has  excess degree at most $2d-6$. Let $F$ denote the base, which has $2d+1$ vertices. The apex of the pyramid has excess degree $d+1$, and so $F$ has excess degree at most $(2d-6)-(d+1)$. Since $d-7<d-2$, the Excess Theorem informs us that $F$ is simple and $d=7$. By \cref{lem:simple}, the only simple 6-polytope with $15=3\times6-3$ vertices is $\Delta_{2,4}$.
\end{proof}

A fundamental property of polytopes is that for any facet $F$ and any ridge $R$ contained in $F$, there is a unique facet $F'$ containing $R$ and different from $F$. In this situation, we have $R=F\cap F'$, and we will $F'$ the {\bf other facet} for $R$.

\begin{lemma}\label{lem:cutface}\cite[Theorem 15.5]{Bro83} Suppose $P$ is a  $d$-polytope and $F$ is a proper face of $P$. Then the subgraph of the graph of $P$ induced by the vertices outside $F$ is connected.
\end{lemma}

\begin{lemma}\label{lem:3out} Suppose $P$ is a  $d$-polytope,  $F$ is a facet of $P$, and there are precisely three vertices, say $u_1,u_2,u_3$, outside $F$,  all of them simple. Then, at least $d-4$  ridges contained in $F$ have the property that their other facet contains all three of $u_1,u_2,u_3$.
\end{lemma}

\begin{proof}
By \cref{lem:cutface}, the subgraph containing $u_1,u_2,u_3$ is connected. There are two cases to consider. Either the three vertices are mutually adjacent, and each is adjacent to exactly $d-2$ vertices in $F$. Or (after relabelling) $u_1$ is not adjacent to $u_3$, in which case $u_1$ and $u_3$ are both adjacent to $u_2$ and to $d-1$ vertices in $F$, while $u_2$ is adjacent to $d-2$ vetices in $F$.

In the first case, each of the three vertices will have degree $d-1$ in any facet which contains it; such a facet must therefore contain one of the other two. So no ridge in $F$ has the property that its other facet contains precisely one of $u_1,u_2,u_3$.

Again by the simplicity of $u_1$, there is only one facet $F'$ of $P$ which contains $u_1$ and $u_2$ but not $u_3$. Thus there is at most one ridge in $F$ whose other facet is $F'$. Hence there are at most three ridges in $F$ having the property that their other facet contains precisely two of $u_1,u_2,u_3$.

Since $F$ contains at least $d$ ridges, at least $d-3$ of them must have the alleged property.

In  the second case, the same reasoning shows that every facet containing $u_2$ also contains either $u_1$ or $u_3$; there is precisely one facet containing $u_1$ and $u_2$ but not $u_3$, and precisely one facet containing $u_2$ and $u_3$ but not $u_1$. There will also be precisely one facet containing $u_1$ but not $u_2$ or $u_3$, and precisely one facet containing $u_3$ but not $u_1$ or $u_2$. No facet can  contain $u_1$ and $u_3$ but not $u_2$.  Hence there are at most four ridges in $F$ having the property that their other facet contains either one or two of $u_1,u_2,u_3$.

Since $F$ contains at least $d$ ridges, at least $d-4$ of them must have the alleged property.

\end{proof}

A common situation for us will be the need to estimate the number of edges  involving a particular set of vertices (often, but not always, the complement of a given facet).

\begin{lemma}\label{lem:outside}\cite[Lemma 4]{PinUgoYos19} Let $S$ be a set of  $n$   vertices of  a $d$-polytope $P$, with $n\le d$.
Then the total number of edges containing at least one vertex in $S$ is at least
$nd-{n \choose 2}$. Moreover, this minimum is obtained precisely when every vertex in $S$ is simple, and every two vertices in $S$ are adjacent.
\end{lemma}

\begin{proof}
Every vertex in $S$ has degree at least $d$, and so is adjacent to at least $d-n+1$ vertices not in $S$. Thus the number of edges between a vertex in $S$ and a vertex $outside$ is at least $n(d-n+1)$.
\end{proof}

We need some information about the structure of $d$-polytopes with $2d$ vertices and whose number of edges is close to minimal. It is known that such a polytope has $d^2$ edges only if it is a prism, and $d^2+1$ edges only if $d=3$ \cite[Theorem 13]{PrzYos16}.

\begin{lemma}\label{lem:minplus2edges} Let $P$ be a $d$-polytope with $2d$ vertices and  $d^2+2$  edges. Then $P$ is one of only seven examples, all with dimension at most five. More precisely
\begin{enumerate}[(i)]
\item If $d=5$,  $P$ is a pyramid over $\Delta_{2,2}$.
\item If $d=4$,  $P$ has at least two nonsimple vertices, and is a pyramid over either a pentasm or $\Sigma_3$, or one of the two polytopes detailed in the table below.
\item For $d=3$, $P$ is the dual of either a pentasm or $\Sigma_3$.
\end{enumerate}
\end{lemma}

\begin{proof}
A special case of \cite[Thm. 19]{PinUgoYos19} asserts that  a $d$-polytope with $2d$ vertices which is not a prism must have at least $d^2+d-3$ edges, and this is $\ge d^2+3$ if $d\ge6$. Thus $d<6$.

(i) If $d=5$, then $P$ has 27 edges and excess degree $4=d-1$, so \cite[Theorem 4.18]{PinUgoYos18} informs us that $P$ is a Shephard polytope, in particular either decomposable or a pyramid. But $P$ is not a prism, so must be indecomposable by \cref{decomp}(vi). Thus $P$ is a pyramid, and its base must have nine vertices and 18 edges, making  $\Delta_{2,2}$  the only option.

(ii) In the case $d=4$, $P$ has eight vertices, 18 edges and excess degree four. Any vertex of $P$ has degree at most seven and hence excess degree at most three. So there must be at least two nonsimple vertices. It is not hard to establish directly that there are only  four examples, but we   simply note that this can be verified from catalogues such as \cite{FukMiyMor13}. Let us now describe these four examples. Two obvious examples are  the pyramid over a pentasm and the pyramid over  $\Sigma_3$, which both have just two nonsimple vertices.

Another ``well known" example is  given in \cite[Figure 1d]{PinUgoYosxix}, which has three nonsimple vertices; its facets are one prism, two tetragonal antiwedges, one pyramid and three simplices; and the vertex-facet relations are detailed in the first column in \cref{tab:4Polytopes8Vertices18Edges}. For a concrete representation, take the convex hull of
$(\varepsilon, 0, 0, 0)$,
$(1, 0, 0, 0)$,
$(0, 1, 0, 0)$,
$(0, 0, 1, 0)$,
$(0, 0, 0, 1)$,
$(1, 0, 0, 1)$,
$(0, 1, 0, 1)$, and
$(0, 0, 1, 1)$,
where $\varepsilon >0$ need not be too small.

The fourth example has a straightforward concrete representation, with vertices $(\pm1, \pm1, 0, 0)$, $(\pm1, 0, 1, 0)$, and $(0, \pm1, 0, 1)$. It is not hard to verify that its facets are two prisms, four quadrilateral pyramids and a simplex. There are four nonsimple vertices, and the vertex-facet relations are detailed in the second column in \cref{tab:4Polytopes8Vertices18Edges}.

\begin{table}
\begin{tabular}{c c c }
{Facet}&{Polytope 1}&{Polytope 2}\\
\hline
1:&\{1 2 3 4 5 6\}&\{1 2 3 4 5 6\}\\
2:&\{1 2 3 4 7 8\}&\{1 2 3 4 7 8\}\\
3:&\{1 2 5 6 7 8\}&\{1 2 5 6 7\}\\
4:&\{3 4 5 6 7\}&\{1 3 5 7 8\}\\
5:&\{2 4 6 8\}&\{ 3 4 5 6 8\}\\
6:&\{1 3 5 7\}&\{2 4 6 7 8\}\\
7:&\{4 6 7 8\}&\{5 6 7 8\}\\
\hline
\end{tabular}
\caption{Vertex-facet incidences of nonpyramidal 4-polytopes with eight vertices and  eighteen edges. }
\label{tab:4Polytopes8Vertices18Edges}
\end{table}

(iii) Suppose $d=3$. Then $P$ has six vertices and 11 edges, so  Euler's formula ensures that its dual $P^*$ must have seven vertices and 11 edges. Thus $P^*$ is either a pentasm or $\Sigma_3$.
\end{proof}

\begin{lemma}\label{lem:nopentasm} Let $P$ be a $d$-polytope with $2d+2$ vertices and no more than $d^2+2d-3$  edges. Suppose no facet of $P$ has $2d$ vertices and no facet of $P$ is a $(d-1)$-pentasm, but that some facet has $2d-1$ vertices. Then either $d=7$ and  $P$ is a pyramid over $\Delta_{2,4}$, or $d=4$ and $P$ is $C_4$ or $\Sigma_4$.
\end{lemma}

\begin{proof}
In three dimensions, a facet with five vertices is a pentasm, so no polytope satisfies the hypotheses. If $d=4$, we simply recall  \cite[Theorem 6.1]{PinUgoYos19}, which asserts that the only 4-polytopes with 10 vertices and no more than 21 edges are $A_4, B_4, C_4$, and $\Sigma_4$. Note that both $A_4$ and $B_4$ contain facets with eight vertices, and do not satisfy the hypotheses.

So assume $d\ge5$. We will show that $P$ is a pyramid. If $F$ is a facet having $2d-1=2(d-1)+1$ vertices, \cref{thm:2d+1} ensures that $F$ has at least $d^2-d$ edges, and consequently is not simple. Moreover, the three vertices $u_1,u_2,u_3$ outside $F$ must be mutually adjacent and all simple, and $F$ must have exactly $d^2-d$ edges; otherwise by \cref{lem:outside}, $P$ would have strictly more than $d^2+2d-3$  edges.

\cref{lem:3out} ensures that some ridge $R$ contained in $F$ is such that its other facet $F'$ contains all three of $u_1,u_2,u_3$. Each $u_i$ is simple in $F'$, and so has $d-3$ edges running into $R$, two edges running into the other $u_j$ and exactly one edge running into  $P\setminus F'$. But each vertex in $F\setminus R$ is adjacent to at least one $u_i$; this implies that $P\setminus F'$ contains at most three vertices. If $P$ is a pyramid over $F'$, \cref{lem:pyramid} completes the proof. Otherwise, $F'$ has $2d-1$ vertices and $P\setminus F'$ contains exactly three vertices. Then there are exactly three vertices in $F\setminus R$, say $w_1,w_2,w_3$, each of them simple, and three edges joining them to $u_1,u_2,u_3$. Then $R$ has $2d-4$ vertices, and there are $3(d-3)$ edges between $u_1,u_2,u_3$ and $R$, likewise at least  $3d-9$ edges between $w_1,w_2,w_3$ and $R$, and nine edges between $u_1,u_2,u_3,w_1,w_2,w_3$. Thus the number of edges in $R$ is exactly
$(d-2)^2+2=\phi(2d-4,d-2)+2$. According to \cref{lem:minplus2edges}, this implies that $d-2<6$.

Since every vertex in $F\setminus R$ is simple, \cref{lem:simpleVertexOutside} ensures that every vertex in $R$ which is not simple in $F$ has at least two neighbours in $F\setminus R$. This implies that the number of nonsimple vertices in $R$ is at most $(3d-9)-(2d-4)=d-5$. Since $F$ is not simple, $R$ must contain a nonsimple vertex, whence $d\ge6$. But if $d=6$, then $R$ is 4-dimensional with eight vertices, 18 edges, and a unique nonsimple vertex, which is impossible by \cref{lem:minplus2edges}(ii).

The only remaining possibility is that $d=7$. By \cref{lem:minplus2edges}(i), the only 5-polytope with 10 vertices and 27 edges is the pyramid over $\Delta_{2,2}$; this must be $R$. The apex of $R$ will then be the only nonsimple vertex in $P$. With excess degree eight, it must be adjacent to every other vertex in $P$. A special case of \cite[Corollary 2.2]{PinUgoYos19} asserts that a  polytope with a unique nonsimple vertex, which is adjacent to every other vertex, must be a pyramid. Again, the base can  be only $\Delta_{2,4}$.
\end{proof}

\begin{lemma}\label{lem:twofaces} Let $P$ be a  $d$-polytope,  with two disjoint faces $F_1$ and $F_2$ whose union contains  ${\rm Vert(P)}$. Suppose that $F_1$ is a facet, and that $F_2=[w_0,w_1]$ is an edge. Then $P$ is decomposable if, and only if, $F_1$ has Shephard's property in $P$. In this case, denoting $V_i=\{u\in {\rm Vert(P)}\cap F: u$ is adjacent to $w_i\}$,  every vertex in $V_0$ is adjacent to at most one vertex in $V_{1}$ and vice versa. Moreover, $F_1$ is also decomposable.
\end{lemma}

\begin{proof}
Any facet disjoint from $F_2$ must be contained in, and hence equal to, $F_1$.

If $F_1$ fails Shephard's property, there will be a vertex $u\in F_1$ adjacent to both vertices in $F_2$. The resulting triangle will be an indecomposable graph touching every facet, which implies indecomposability of $P$ by \cref{decomp}.

If $F_1$ has Shephard's property, then $P$ is decomposable, according to \cref{shp}.
Now suppose there is a vertex $a\in V_0$ which is adjacent to two distinct vertices $b,c\in V_1$. It follows that the five vertices $a,b,c, w_0,w_1$ are not contained in any plane. The graph $G$ comprising these six edges is isomorphic to the complete bipartite graph $K_{3,2}$. According to \cref{decomp}(iii), $G$ is indecomposable. By the remark at the beginning of this proof, $G$ touches every facet, contradicting the decomposability of $P$. Decomposability of $F_1$ now follows from \cite[Lemma 5.5]{PinUgoYos18}, but we repeat the short argument: the graph of $F_1$ also touches every facet, so \cref{decomp}(iv) is again applicable.
\end{proof}

We remark that in \cref{lem:twofaces}, the edge $F_2$ will actually be a summand of a polytope combinatorially equivalent to $F_1$. This can be deduced from the proof of \cite[Proposition 5]{PrzYos16}. Moreover, a generalisation of this result remains valid when $F_2$ is merely assumed to be an indecomposable face. We don't need these stronger versions, so we omit the details.

The next result improves \cite[Lemma 5.6(ii) and Remark 5.7]{PinUgoYos18}.

\begin{corollary}\label{cor:2out}
Let $F$ be a facet of a  $d$-polytope $P$,  with only two vertices $u_0,u_1$ of $P$ being outside $F$. Suppose $F$  is either $C_4$, $\Sigma_4$,  a pyramid over  $\Delta_{2,4}$, or $\Delta_{m,n}$ with $m\ge n\ge2$. Then $F$ fails Shephard's property in $P$, and $P$ is indecomposable. In case $F$ is  $\Delta_{m,n}$, there are at least $2(m+1)n$ edges running out of $F$.
\end{corollary}

\begin{proof} The two vertices outside $F$ must be adjacent, and so constitute an edge. By \cref{lem:twofaces}, failure of Shephard's property for $F$ is equivalent to indecomposability of $P$.

For the case of a pyramid over  $\Delta_{2,4}$, $F$ is indecomposable by \cref{decomp}(v), and then $P$ is indecomposable by \cref{decomp}(iv).

Suppose next that $F$ is  $C_4$ or $\Sigma_4$ . Inspection of the graphs (\cref{c4sigma4}) reveals that in both cases there are four  triangles $T_1,T_2,T_3,T_4$, with $T_i\cap T_{i+1}$ nonempty for $i=1,2,3$, whose union contains at least seven vertices.  Each of the three (or fewer) remaining vertices must then be adjacent to at least two vertices in this collection of triangles.
If $F$ had Shephard's property, then \cref{lem:twofaces} would allow us to  colour the vertices of $F$ with two colours in such a way that every vertex  is adjacent to at most one vertex of the other color. In particular, any three mutually adjacent vertices would have the same colour. In our situation, all vertices of $F$ would have the same colour, i.e. no such 2-coloring is  possible.

\begin{figure}[h]

\includegraphics[scale=1]{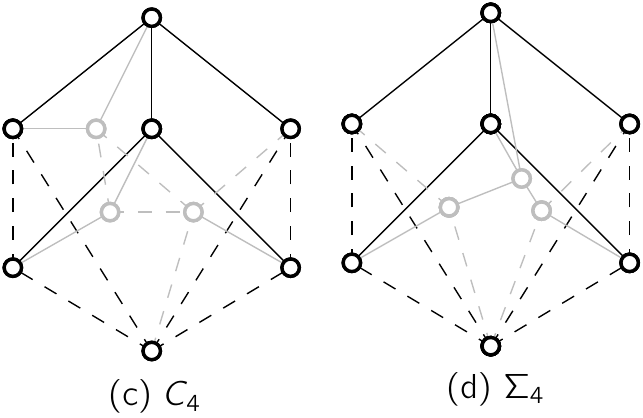}
\caption{Polytopes $C_4$ and $\Sigma_4$. In each polytope, four triangles $T_1,T_2,T_3,T_4$ with $T_i\cap T_{i+1}$ nonempty for $i=1,2,3$ are highlighted in dashed lines.}\label{c4sigma4}
\end{figure}

Finally, suppose $F=\Delta_{m,n}$; then $m+n=d-1$. We claim that there are least $2(m+1)n$ edges between $F$ and $u_0,u_1$.
Fix a facet $F'$ of  $P$ containing $u_0$ but not $u_1$. Let $R$ be an arbitrary ridge contained in $F'$ but not containing $u_0$. Clearly $R\subset F$ and $R\subset F'$, which forces $R=F\cap F'$. Thus $R$ is the unique ridge of $F'$ not containing $u_0$. This implies $F'$ is a pyramid over $R$ and $u_0$ is adjacent to every  vertex in $R$.  Being contained in $F$, $R$ must be either of the form $\Delta_{m-1,n}$, with $mn+m$ vertices, or of the form $\Delta_{m,n-1}$, with $mn+n$ vertices. So $u_0$ has at least $(m+1)n$ edges running into $F$. Likewise for $u_1$. Since $n>1$, we have $2mn+2n>(m+1)(n+1)=v(F)$, meaning that $F$ fails Shephard's property.\end{proof}

We  now have enough machinery to prove \cref{thm:2d+2}. We reformulate it slightly to streamline the proof.

\begin{theorem}\label{thm:2d+2rewritten} For $d\ge3$, the only $d$-polytopes with $2d+2$ vertices and $d^2+2d-3$ or fewer edges are $A_d$, $B_d$, $C_4$, $\Sigma_4$, $\Delta_{2,3}$, and the pyramid over $\Delta_{2,4}$. Each of these examples, except $\Delta_{2,3}$, has precisely $d^2+2d-3$ edges.
 \end{theorem}

\begin{proof} The case $d=3$ is well known and easy to prove: $A_3$ is the cube and $B_3$ is the 5-wedge. See also \cite{BriDun73} or \cite{Fed74}. The case $d=4$ was established in \cite[\S6]{PinUgoYos18}. We proceed by induction on $d$, henceforth assuming $d\ge5$.

Let $P$ be a $d$-polytope with $2d+2$ vertices, by hypothesis with excess degree at most $2d-6$.
We distinguish a number of cases based on the maximum number of vertices of the facets in $P$.
We will show in each case that $P$ either has one of the required forms or  has strictly more than $d^2+2d-3$ edges.

{\bf Case 1: Some facet has $2d+1$ vertices.}
In this case, $P$ is a pyramid and \cref{lem:pyramid} informs us that the base is  $\Delta_{2,4}$.

{\bf Case 2: $P$ is not a pyramid, but some facet $F$ has $2d=2(d-1)+2$ vertices.}

There are two adjacent vertices outside $F$, which we will  call $u_1$ and $u_2$, and at least $2d$ edges running out of $F$. Thus $F$ has at most $d^2+2d-3 -(2d+1)=(d-1)^2+2(d-1)-3$ edges.  By induction, $F$ must be $A_{d-1}$, $B_{d-1}$, $\Delta_{2,3}$, $C_4$, $\Sigma_4$ or a pyramid over $\Delta_{2,4}$. If $F$  were $\Delta_{3,2}$,  \cref{cor:2out} would ensure  at least 16 edges running out of $F$, giving $P$ at least $47>6^2+2\times6-3$ edges. In all the other cases, $F$ has exactly $(d-1)^2+2(d-1)-3$ edges, and so there are exactly $2d$ edges running out of $F$. Thus $F$ has Shephard's property, and \cref{cor:2out} rules out $C_4$, $\Sigma_4$ and the pyramid over $\Delta_{2,4}$ as options.

So $F$ is of the form $A_{d-1}$ or $B_{d-1}$: Recall their structure from \cref{rmk:Ad-Facets} and  \cref{rmk:Bd-Facets}. In both cases, there are $d+2$ ridges (of $P$) contained in $F$ and $d-4$ of them have same form, i.e. either $A_{d-2}$ or $B_{d-2}$ respectively. If the other facet $F'$ for one of these $d-4$ ridges were a pyramid, its apex, say $u_1$, would be adjacent to all $2d-2$ vertices in the ridge, giving $d^2-4+2d-2$ edges in the union of the two facets. But $u_2$ has degree at least $d$, which would give $P$ at least $d^2+3d-6>d^2+2d-3$ edges. So each such ``other facet" must contain both $u_1$ and $u_2$, i.e. $F'$ has $2(d-1)+2$ vertices.  It follows that $F'$ also has at most $(d^2+2d-3)-(2d+1)$ edges. The induction hypothesis tells us that $F'$ is also of the form $A_{d-1}$ or $B_{d-1}$, respectively.


Let $S_i$ and $Q_i$ be as in \cref{rmk:AdBd}.
The $d-4$ ridges referred to above each omit one of the $d-4$ edges linking $S_1$ and $S_2$. Given that their other facets have the same form as $F$ ($A_{d-1}$ or $B_{d-1}$ respectively), we can suppose that $u_1$ is adjacent to every vertex in $S_1$ and $Q_1$, and that $u_2$ is adjacent to every vertex in $S_2$ and $Q_2$. Thus $P$ has the same graph as $A_{d}$ or $B_{d}$. We claim this  ensures that $P$ is $A_d$ or $B_d$ respectively.

If $F$ is $A_{d-1}$, the  six ridges of $P$ contained in $F$ which are not of the form $A_{d-2}$ correspond to the six faces of the cube, i.e. each is the convex hull of one face of the cube and the $(d-3)$-prism. They are
 two copies of $M_{2,d-4}$ (the convex hull of $Q_i\cup S_i$ for $i=1,2$);
and four $(d-2)$-prisms (the convex hull of $S_1\cup S_2\cup E$, for each of the four edges $E$ linking $Q_1$ and $Q_2$). The other facets for these ridges are now easy to see. The other facet corresponding to each copy of $M_{2,d-4}$ is copy of $M_{2,d-3}$, namely the convex hull $Q_i, S_i $ and $u_i$. For each prism facet, the other facet is a $(d-2)$-prism, containing both $u_1$ and $u_2$. This completely describes the facet-vertex incidences of $P$; it is the same as $A_d$.

Likewise, if $F$ is $B_{d-1}$, the other six ridges of $P$ contained in $F$ correspond to the six faces of the 5-wedge. They are one copy of $M_{2,d-4}$, one $(d-2)$-prism, two simplices and two pentasms. A similar argument investigating their other facets determines the facet-vertex incidences. (Alternatively, since $P$ has only $d-3$ nonsimple vertices, we could apply \cite[Theorem 3.1]{DooNevPinUgoYos19}, which asserts that the face lattice in this case is determined by the 2-skeleton.)

{\bf Case 3A: No facet has $2d$ or more vertices, and some facet $F$  is a $(d-1)$-pentasm.}

We show that this case cannot arise. Recall that a $(d-1)$-pentasm has $2d-1=2(d-1)+1$ vertices and $d^2-d-1=(d-1)^2+(d-1)-1$ edges.

One of the $(d-2)$-faces of $F$ is a $(d-2)$-pentasm $R$, which has $(d-2)^2+(d-2)-1$ edges; there will then be $2d-2$ edges incident with vertices in $F\setminus R$, and one of the vertices in $F\setminus R$ will be nonsimple. We consider the other facet $F'$ of $P$ containing $R$.  As a pentasm, $R$ has excess degree $d-4$. If $F'$ were a pyramid, its apex would have degree $2d-3$, and hence excess $d-2$, in $F'$. This would give $F'$ excess degree $2d-6$. Now the excess degree of any facet of a nonsimple polytope is strictly less than the excess degree of the whole polytope \cite[Lemma 3.2]{PinUgoYos18};  thus $P$ would have excess degree at least $2d-5>2d-6$, contrary to hypothesis. So we can assume that $F'$ is not  a pyramid.

The facet $F'$ must contain only two of the three vertices $t, u,w$ outside $F$, say $t,u$, otherwise it would  contradict our choice of $F$.  Furthermore, the vertices $u,t$ must be adjacent. Then by \cref{thm:2d+1},
$F\cup F'$ contains at least $(d^2-d-1)+(2d-2)=d^2+d-3$ edges; note that this total does not include any edges between $F\setminus R$ and $F'\setminus R$. If the vertex $w$ outside  $F\cup F'$ is not simple, then $P$ will have at least  $(d^2+d-3)+(d+1)=d^2+2d-2$ edges, contrary to hypothesis. So we assume that $w$ is simple.

Clearly  the nonsimple vertex in $F\setminus R$ is adjacent to some vertex vertex outside $F$; if it is not adjacent to $w$, then it must be adjacent to $u$ or $t$. But if it is adjacent to $w$, it must, by \cref{lem:simpleVertexOutside}, be adjacent to another vertex outside $F$, which can only be one of $t,u$. In either case, this gives us an edge between $F\setminus R$ and $F'\setminus R$. Of course there are $d$ edges containing $w$, so again $P$  has at least $(d^2+d-3)+(d+1)=d^2+2d-2$ edges.

{\bf Case 3B: No facet has $2d$  vertices, no facet is a pentasm, but some facet $F$ has $2d-1=2(d-1)+1$ vertices.}

This case is settled by \cref{lem:nopentasm}.

{\bf Case 4A. No facet has $2d-1$ or more vertices, some facet $F$ has $2d-2=2(d-1)$ vertices, but is not a prism.}

We show that this case cannot arise; our argument  depends on the dimension. First note that $F$ has at least $(d-1)^2+(d-1)-3$ edges \cite[Theorem 19]{PinUgoYos19}, and by \cref{lem:outside} the four vertices outside $F$ belong to at least $4d-6$ edges. Thus $P$ has at least $d^2+3d-9$ edges, and for $d>6$, this exceeds $d^2+2d-3$.

If $d=6$, and $F$ has 28 edges or more, then $P$ has at least 46 edges. Otherwise,  $F$ has $27=5^2+2$ edges and must, by \cref{lem:minplus2edges}(i),  be a pyramid over $\Delta_{2,2}$. The other facet containing the ridge $\Delta_{2,2}$ has only only 10 vertices, so must also be a pyramid. Then these two facets have 36 edges between them, and the three remaining vertices must belong to at least 15 edges. This gives $P$ at least 51 edges.

Finally if $d=5$, then $F$ will have eight vertices and hence more than sixteen edges. It cannot have seventeen edges, thanks to \cite[Thm. 10.4.2]{Gru03}; see also \cite[remarks after Prop. 2.7]{PinUgoYosxix}. So $F$ has at least eighteen edges, and excess degree at least four. The excess degree of $P$ is strictly greater than that of $F$ \cite[Lemma 3.2]{PinUgoYos18}, and even, hence at least $6>2d-6$, contrary to hypothesis.

{\bf Case 4B. No facet has $2d-1$ or more vertices, some facet $F$ has $2d-2=2(d-1)$ vertices, and every such facet is a prism.}

We show that this case  does not arise either. Suppose $F$ is a simplicial $(d-1)$-prism. We can label its vertices as $\{u_1,\cdots,u_{d-1},w_1,\cdots,w_{d-1}\}$, so that  $u_i$ is adjacent to $u_j$ and $w_i$ is adjacent to $w_j$ for all $i,j$, but $u_i$ is adjacent to $w_j$ if and only if $i=j$. The graph and face lattice are clear from this notation.
Then the ridge  $R$ contained in $F$ with vertices  $\{u_1,\cdots,u_{d-2},w_1,\cdots,w_{d-2}\}$ is a simplicial $(d-2)$-prism. Consider the other facet $F'$ containing $R$.  If $F'$ were a pyramid, then $F\cup F'$ would contain $(d-1)^2+2(d-2)=d^2-3$ edges, and by \cref{lem:outside} the remaining three vertices would belong to at least $3d-3$ edges. But then $P$ would have at least $d^2+3d-6>d^2+2d-3$ edges. Then $F'$ must be another prism. We can label the two vertices of $F'\setminus R$ as $u_0$ and $w_0$, with the adjacency relationships clear from the notation. Now the graph of $F\cup F'$ has $d^2-2$ edges, and by \cref{lem:outside} again, there are at least $2d-1$ edges involving the other two vertices, say $a,b$. Since $d^2-2+2d-1=d^2+2d-3$, there must be precisely $2d-1$ edges involving $a$ and $b$, and no edge between $u_0$ and $u_{d-1},w_{d-1}$, nor between $w_0$ and $u_{d-1},w_{d-1}$.
In particular, $u_0,w_0,u_{d-1},w_{d-1}$ each have degree $d-1$ in $F\cup F'$, so each of them must be connected to either $a$ or $b$. On the other hand, there are at most $2d-1$ vertices in $F\cup F'$ adjacent to either $a$ or $b$. Without loss of generality, we suppose that $u_{d-3}$ is adjacent neither to $a$ nor to $b$.

Now consider the ridge $R'$ in $F'$ with vertices  $\{u_0,\cdots,u_{d-3},w_0,\cdots,w_{d-3}\}$. The other facet $F''$ containing $R'$ must be a prism. It cannot contain $u_{d-2}$ or $w_{d-2}$ because they belong to $F'$. Nor can it contain $u_{d-1}$ or $w_{d-1}$, because they are not adjacent to $u_0$ and $w_0$. And it cannot contain $a$ or $b$, because they are not adjacent to $u_{d-3}$. This is absurd.

{\bf Case 5. Some facet $F$ has between $d+3=d-1+4$ and $2d-3=d-1+d-2$ vertices, but no facet has $2d-2$ or more vertices.}

Denote by $n$ the number of vertices outside $F$; then  $5\le n\le d-1$. Then the facet $F$ has $(d-1)+(d-n+3)$ vertices and $d-n+3<d-1$, so \cite[Theorem 8(iii)]{PinUgoYos19} ensures that $F$ has excess at least $(d-n+2)(n-3-1)$.  Thus, the total number of edges in $F$ is at least $\h((d-1)(2d+2-n)+(d-n+2)(n-4))$, with equality only if $F$ is a triplex. \cref{lem:outside} informs us that the total number of edges outside $F$ is at least $\h(2nd-n^2+n)$. Thus the number of edges in $P$ is at least $d^2 +(n-2)d-(n^2-4n+5)$, which is $>d^2+2d-3$ provided
$$d>n+\frac{2}{n-4}.$$
If $n\ge7$, then $n+\frac{2}{n-4}<n+1\le d$, so we are fine. If $n=6$, we need to consider the case $d=7$ separately. If $n=5$, we need to consider the cases $d=6,7$ separately.

Note that $d-n+3=4$ or 5 in all three cases. If $F$ is not a triplex, then \cite[Theorems 19 and 20]{PinUgoYos19} guarantee $F$ has at least $\phi(2d+2-n,d-1)+2$ edges. So the total number of edges in $P$ is at least
$d^2 +(n-2)d-(n^2-4n)-3>d^2+2d-3$.

If $F$ is a triplex, it is a pyramid over some ridge $R$ which is also a triplex. By maximality, its other facet $F'$ must also be a triplex. Now $R$ has $(d-2)+(d-n+3)$ vertices, and hence $\phi(2d-n+1,d-2)=d^2-4d-\h( n^2-9n+12)$ edges. The two apices of $F,F'$ belong to $2(2d-n+3)$ edges in $F\cup F'$. The $n-1$ vertices outside $F\cup F'$ must belong to at least $(n-1)d-{n-1\choose2}$ edges. Adding these up gives a total of at least $d^2+(n-1)d-(n^2-4n+1)$ edges, and this exceeds $d^2+2d-3$ in each case of interest, namely $(n,d)=(6,7), (5,6)$, and $(5,7)$.

{\bf Case 6. Some facet $F$ has $d+2=d-1+3$ vertices, and no facet has more vertices.}

First suppose $F$ is a pyramid. Then the base is a ridge $R$ with exactly $d+1=d-2+3$ vertices, and so has at most six missing edges. The other facet $F'$ containing $R$ must also be a pyramid over $R$, and so the union of the two facets $F\cup F'$ has at most seven missing edges. The union has at least ${d+3\choose2}-7=\h(d^2+5d)-4$ edges. The $d-1$ vertices outside $F\cup F'$ belong to at least $(d-1)d-{d-1\choose2}=\h(d^2+d)-1$ edges. Thus $P$ has at least $d^2+3d-5>d^2+2d-3$ edges.

Next suppose that $F$ is simplicial. The Lower Bound Theorem ensures that $F$ has at least ${d+2\choose2}-3=\h(d^2+3d)-2$ edges. The $d$ vertices outside $F$ belong to at least $d^2-{d\choose2}=\h(d^2+d)$ edges. Thus $P$ has at least $d^2+2d-2$ edges.

If $F$ is neither simplicial nor a pyramid, then some ridge $R$ in $F$ has exactly $d$ vertices. Let $F'$ be the other facet containing $R$. We need to distinguish two cases, depending whether $F'$ has $d+1$ or $d+2$ vertices. In either case, $F$ has at least $\phi(d-1+3,d-1)=\h(d^2+3d)-5$ edges.

If $F'$ has $d+1$ vertices, it is a pyramid over $R$, and its apex belongs to $d$ edges in $F'$. Again by \cref{lem:outside}, the $d-1$ vertices outside $F\cup F'$ belong to at least $d(d-1)-{d-1\choose2}=\h(d^2+d)-1$ edges. Thus $P$ has at least $d^2+3d-6>d^2+2d-3$ edges.

If $F'$ has $d+2$ vertices, there are two vertices in $F'\setminus F$, which must belong to at least $2d-3$ edges in $F'$. Moreover, the $d-2$ vertices outside $F\cup F'$ belong to at least $d(d-2)-{d-2\choose2}=\h(d^2+d)-3$ edges. Thus $P$ has at least $d^2+4d-11>d^2+2d-3$ edges.

{\bf Case 7. Some facet $F$ has $d+1=(d-1)+2$ vertices, and no facet has more vertices.}

First consider the case when $F$ is simplicial; then it has at most one missing edge.  Let $R$ be any ridge in $F$, and let $G$ be the other facet facet containing $R$. Of course $R$ is a simplex, while $F'$ may have either $d$ or $d+1$ vertices.

If $F'$ has just $d$ vertices, then it is also a pyramid over $R$, whose apex is adjacent to every vertex in $R$ but possibly not adjacent to the two vertices in $F\setminus R$. With at most three missing edges, $F\cup F'$  has least ${d+2\choose2}-3$ edges and \cref{lem:outside} ensures that the $d$ vertices outside $F\cup F'$ belong to at least $d^2-{d\choose2}$ edges. This gives a total of at least $d^2+2d-2$ edges in $P$.

If on the other hand $F'$ has $d+1$ vertices, then the two vertices in $F'\setminus R$ belong to at least $2d-3$ edges in $F'$, giving $F\cup F'$ at least $\h(d^2+5d)-4$ edges. Again,  the $d-1$ vertices outside $F\cup F'$   contribute at least $\h(d^2+d)-1$ edges. Thus $P$ has at least $d^2+3d-5>d^2+2d-3$ edges.

Now suppose that  $F$ is not simplicial; then it is $M_{2,d-3}$. In particular, it contains a ridge $R$ with  $d$ vertices, and so must be a pyramid over $R$. Then $F'$ must have $d+1$ vertices and also be a pyramid. Now $R$ has at least $\phi(d,d-2)=\h(d^2-d)-2$ edges, the two apices belong to $2d$ edges in $F\cup F'$, and the $d$ vertices outside $F\cup F'$ belong to at least $\h(d^2+d)$ edges. This gives a total of at least $d^2+2d-2$ edges.

{\bf Case 8. Every facet $F$ has just $d$ vertices.}

Then $P$ is simplicial and the conclusion follows from the Lower Bound Theorem (\cref{thm:LBT}(: $P$ has at least ${d\choose 1}(2d+3)-{d+1\choose 2}=\frac{3}{2}d^2+\frac{5}{2}d>d^2+2d-3$ edges.

\end{proof}

\section{Polytopes with $2d+3$ or more vertices}

We know the exact value $\min E(v,d)$ for $v\le 2d$. Consider now the possibility of extending this result to polytopes with even more vertices. The problem of minimising  the number of edges, over a family of all $d$-polytopes which all have the same number of vertices, is the same as minimising the excess degree over the same family. Accordingly, we find it convenient to reformulate this problem in terms of the excess degree. Let us define, for $v>d$,
$$f(v)=f_d(v)=\min\{\xi(P): P {\rm\ is\ a\ } d{\rm-polytope\ with\ } v {\rm\ vertices}\}.$$
We will generally suppress the subscript $_d$, regarding the dimension as fixed. We know that $f(v)=(v-d-1)(2d-v)$ whenever $v\le2d$. Moreover, the minimising polytopes for such $v$ all have (at least four) simple vertices.
As we noted earlier, truncating a simple vertex of any polytope yields a new polytope with $d-1$ more vertices than the original, but the same excess degree. Applying this repeatedly to the aforementioned minimisin gpolytopes with $v\le2d$, we see that $f(v+n(d-1))\le (v-d-1)(2d-v)$ for all  $n$ and all $v\le2d$.  In particular, we have
$$f(v)\le (v-2d)(3d-1-v), \qquad{\rm when}\quad 2d\le v\le3d-1$$
and
$$f(v)\le (v-3d+1)(4d-2-v), \qquad{\rm when}\quad 3d-1\le v\le4d-2.$$

A {\it reasonable question} to pose is whether this inequality is actually an equality, at least for small values of $v$. We proved that equality holds for $v=2d+1$ in \cite{PinUgoYos19} (except when $d=4$), and for $v=2d+2$ in this paper (except when $d=5$).

Equality clearly holds when $v=3d-1$ because the expression on the right is zero. It holds when $v=3d$ (except when $d=4$ or 8),  because the expression on the right is then $d-2$, and it cannot be smaller than $d-2$ for any nonsimple polytope. There are simple polytopes with $3d$ vertices only when $d=4$ or 8 \cite[Lemma 2.19(vi)]{PinUgoYos18}.

It holds when $v=3d-2$ (except when $d=6$), again because the expression on the right is $d-2$, and there are simple polytopes with $3d-2$ vertices only when $d=6$ [\cref{lem:simple}].

Equality fails for $v=3d-3$. The expression on the right is then $2d-6$, but the existence of $\Delta_{2,d-2}$ shows that $f(3d-3)=0$.

Equality fails again for $v=3d-5$, at least when $d>7$. The expression on the right is then $4d-20$, but a pyramid over $\Delta_{2,d-3}$ has $3d-5$ vertices and excess $2d-6$. Although we do not know its exact value, we have the estimate $f(3d-5)\le2d-6$.

We have little idea about $v=3d-4$.

A more realistic question is whether the inequality above is actually an equality, for $v$ in the range $[2d+3,3d-6]$. As we have seen, some counterexamples arise in low dimensions, so we formulate the question cautiously, to exclude them. We are most reluctant to call this a conjecture.

{\bf Question.} For each $v$ in the range $2d+3\le v\le3d-6$, is there a constant $D=D_v$ such that, for all dimensions $d\ge D_v$, the equality $f_d(v)=(v-2d)(3d-1-v)$ holds? If so, is a truncation of the triplex $M_{v-d+1,2d-v-1}$ the only minimising polytope?

In the rest of this section, we point out a lower bound for $f_d(v), v\ge2d+3$, admittedly somewhat weaker than we would like. We assume $d>7$, otherwise the range of values under consideration for $v$ will be empty, and our conclusions vacuous.

Recall from \cref{lem:simple}  that a $d$-polytope with between $2d+3$ and $3d-4$ vertices cannot be simple. Indeed, any simple polytope, other than a simplex or prism, has at least $3d-3$ vertices.
It turns out that a similar restriction on the number of vertices applies to polytopes with excess degree $d-2$. We recall that the structure of such polytopes is very special \cite[Theorem 4.10]{PinUgoYos18}: Any $d$-polytope $P$ with excess degree precisely $d-2$ has two facets $F_1$ and $F_2$ such that  $F_1\cap F_2=S$ is a simplex face of dimension either 0 or $d-3$, and the nonsimple vertices of $P$ are precisely the vertices of $S$. Moreover, the facets $F_1$ and $F_2$ are both simple, every vertex in $S$ is adjacent only to vertices in $F_1\cup F_2$; and any two distinct facets of $P$ intersect either in a ridge, the empty set, or $S$.

This leads to the following result, whose proof we omit.

\begin{theorem}
Let $P$ be a $d$-polytope with  excess degree exactly $d-2$. Then the number of vertices of $P$ is either $d+2$, $2d-1$, $2d+1$ or $\ge3d-2$.
\end{theorem}

The idea of the proof is to note that if $P$ has $3d-3$ or fewer vertices, each of the facets $F_1$ and $F_2$ just described  must be either a simplex, a prism, $\Delta_{2,d-3}$ or $\Delta_{3,3}$. There are only a limited number of ways that two such facets can intersect in a single vertex, or a $(d-3)$-simplex. The argument can be tweaked to  characterise all $d$-polytopes with excess $d-2$ and up to $3d-2$ vertices: they are

$\bullet$ $M_{2,d-2}$ (with $d+2$ vertices),

$\bullet$ $M_{d-1,1}$ (with $2d-1$ vertices),

$\bullet$ the pentasm (with $2d+1$ vertices)

$\bullet$   $C_d$, $\Sigma_d$, $N_d$ and $A_4$ (with $3d-2$ vertices)

Furthermore, examples with excess $d-2$ and $3d-1$ vertices exist only in dimension 4; they are precisely the three examples discussed in the proof of \cite[Lemma A2]{PinUgoYos18}.

Finally, we can announce an estimate which is stronger than the Excess Theorem.

\begin{corollary}
For $d\ge7$, and  $2d+3\le v\le3d-4$, we have $f_d(v)\ge d$. That is, any $d$-polytope with  $v$ vertices in this range has excess degree at least $d$. In particular,  it must have at least $\h(v+1)d$ edges. The latter conclusion also holds for a $d$-polytope $P$ with $v=3d-3$ and $d\ge6$, provided $P$ is not the simple polytope $\Delta_{2,d-2}$.
\end{corollary}

\end{document}